\documentclass[twocolumn]{IEEEtran}
\usepackage{amsmath}
\usepackage{amssymb}

\usepackage[breaklinks=true]{hyperref}
\usepackage{graphicx}
 \usepackage{color}
\usepackage[show]{ed}
\usepackage{soul} 
\usepackage{epstopdf}

\newcommand{\Eavg}[1]{\mathrm{E}[#1]}
\newcommand{\sq}[1]{\left[ {#1} \right]}
\newcommand{\tr}[1]{{\textrm {Tr}}\sq{#1}}

\newcommand{\smallfrac}[2]{\mbox{$\frac{#1}{#2}$}}
\newcommand{\half}{\smallfrac{1}{2}}
\newcommand{\bra}[1]{\langle{#1}|}
\newcommand{\ket}[1]{|{#1}\rangle}

\newcommand{\op}[2]{\ket{#1}\bra{#2}}

\newcommand{\dg}{^\dagger}
\newcommand{\Dc}[1]{{\cal D}\sq{#1}}

\newcommand{\Hc}[1]{{\cal H}\sq{#1}}

\newcommand{\beq}{\begin{equation}} 
\newcommand{\eeq}{\end{equation}}
\newcommand{\bqa}{\begin{eqnarray}} 
\newcommand{\eqa}{\end{eqnarray}}

\newcommand{\erf}[1]{Eq.~(\ref{#1})}

\newcommand{\frf}[1]{Fig.~\ref{#1}}

\newcommand{\blk}{\color{black}}
\newcommand{\noncr} {\nonumber\\}

\usepackage{times} 
\usepackage{amsmath} 
\usepackage{amssymb}  
\usepackage{amsfonts}
\usepackage{xspace}
\usepackage{graphicx}
\usepackage{color}
\usepackage{subfigure}

\newtheorem{theorem}            {Theorem}[section]

\newtheorem{sideremark}         [theorem]{Remark}
\newtheorem{sideeg}           [theorem]{Example}
\newtheorem{sideconj}           [theorem]{Conjecture}

\newcommand{\qed} {\hskip 0.2em\lower 0.7ex\hbox{\vbox{\hrule
\hbox{\vrule height 1.2ex\hskip 0.4em\vrule height 1.2ex}
\hrule}}}

\newcommand {\controlSet} {\mathcal{V}}
\newcommand {\controlValueSet} {\mathbf{V}}

\newcommand {\dpd}[2] {\frac{{\partial}^{2} {#1}}{\partial {#2}^{2}}}
\newcommand {\spd}[2]{\frac{\partial {#1}}{\partial {#2}}}

\usepackage{color}

\def \secname {sec:optqstate:2010}
\def \eqnname {eqn:optqstate:2010}
\def \thmname {thm:optqstate:2010}

\def \appendname {appendix:optqstate:2010}
\def \solnot {\theta}
\def \aeq {&=}
\def \adefeq {&:=}
\def \notncsig {v}
\def \notncsigval {v}
\def \costfnone {J_{1}}
\def \ocostfnone {S}

\def \costfntwo {J_{2}}
\def \ocostfntwo {R}

\def \notncsigvalalp {\alpha}

\newcounter{AssumptionCounter}
\stepcounter{AssumptionCounter}

\def \controlBound {\Omega}
\def \oaltcostfn {\mathcal{S}}
\def \targetSet {\mathcal{T}}
\def \thetaset {G}
\def \Largethetaset {\mathcal{G}}
\newcommand{\thn}[1] {\theta_{{#1}}}
\def \eop {\mathop{E}}
\begin{document}
\author{Srinivas Sridharan,Masahiro Yanagisawa, Joshua Combes}
\author{Srinivas Sridharan \thanks{S. Sridharan is with the Research School of Engineering, Australian National University, Canberra, ACT 0200, Australia and the University of California San Diego. srsridharan@ucsd.edu}
    \and
    Masahiro Yanagisawa \thanks{M. Yanagisawa is with the Research School of Engineering, Australian National University, Canberra, ACT 0200, Australia. ym@anu.edu.au.}\and
    Joshua Combes \thanks{J. Combes is with the Research School of Engineering, Australian National University, Canberra, ACT 0200, Australia and the Center for Quantum Information and Control
University of New Mexico MSC07--4220 Albuquerque, New Mexico 87131-0001, USA.  joshua.combes@gmail.com. }
}

\title{
Optimal rotation control for a qubit  subject to  continuous measurement 
}

\date{\today}    

\maketitle

\begin{abstract}
In this article we analyze the optimal control strategy for rotating a monitored qubit from an initial pure state to an orthogonal state in minimum time. This strategy is described for two different  cost functions of interest which do not have the usual regularity properties. Hence, as classically smooth cost functions may not exist, we interpret these functions as viscosity solutions to the optimal control problem. 
Specifically we prove their existence and uniqueness in this weak-solution setting. In addition,  we also give bounds on the time optimal control to prepare any pure state from a mixed state.
\end{abstract} 
%

\section{Introduction}
It is anticipated that devices which make use of quantum effects will have a strong impact on future technology \cite{nielsen2000qca,WisMil10}. 
This motivates research on the time optimal procedure for the preparation of any desired state for monitored open quantum systems. 
 For pure states, state preparation can be performed by unitary control whose purpose is to \lq rotate\rq\, the resulting pure state to a target (pure) state as fast as possible. A significant amount of work has been done on this problem for closed quantum systems \cite{Nielsen2006,nielsen2005gaq,caneva2009optimal,carlini2006toq}. At present little research has been undertaken for the corresponding open system problem\footnote{One exception is Ref.~\cite{bouten2005bellman} where a time optimal control problem for monitored open quantum systems was considered for a special case where measurement does not provide information about the system. In this situation the quantum trajectory becomes a linear quantum trajectory; which enabled the authors of \cite{bouten2005bellman} to formulate and then solve the system as a linear quadratic type optimal control problem.}.
Thus there is a need to consider more general cases of the time optimal state preparation problem for control theoretic and application driven reasons. 

In this article we first consider the time optimal ``unitary control'' for a qubit undergoing continuous measurement. To
measure the speed of convergence, we examine two different cost functions for this unitary control stage. The first 
is the mean of the times at which each trajectory attains the target (i.e. the expectation of the stopping time, which is the first passage time to the target state), termed the {\em mean hitting time}, and the other is the time at which the ensemble average of trajectories reaches attains the target state termed the {\em expected trajectory hitting time}.
The Hamilton-Jacobi-Bellman equations that arise in these cases turn out to be degenerate; due to 
this  the uniqueness of the solutions of these equations is not guaranteed. Hence our objective in 
this article is to  obtain  the optimal control strategies for these  quantum control problems  while 
dealing with these degeneracies. 

In the sections that follow, we consider the control problems arising from both these costs and indicate their solution using the dynamic programming approach. Further, we also point out the need for the interpretation of these solutions by a generalized solution framework.
%

\subsection{The problem}
Consider a quantum bit, called a {\em qubit} \footnote{Physically qubits can be realized as two level atoms, see \cite[Chap. 7]{nielsen2000qca} for more details and examples.}. { An arbitrary qubit state, denoted by the $2\times 2$ matrix $\rho$, is a positive operator in Hilbert space with the constraint that its trace is one. In the Dirac notation, pure states -- those states with $\tr{\rho^2}=1$ -- are denoted by a complex vector $\ket{\psi}=(\alpha,\beta)^T$ such that $|\alpha|^2+|\beta|^2=1$. The entire state space of qubits can be represented as a ball of radius one, called a Bloch sphere, where the pure states lie on the surface and the mixed states -- those states with $\tr{\rho^2}<1$ -- are in the interior. The $(x,y,z)$ axis of this sphere correspond to the directions of the eigenvectors of the Pauli matrices $\sigma_{k}$ where $k\in \{x,y,z\}$ \footnote{The Pauli matrices are: $\,\sigma_x = [0 ,1 ; 1 ,0],\,\,\sigma_y = i [0 ,-1 ; 1, 0],\,\, \sigma_z = [1 ,0 ; 0 ,-1]$.}. The states in the $z$ direction are usually denoted by $\ket{0} = (1,0)^T$ (or up in the $z$ direction) and $\ket{1} = (0,1)^T$ (or down in the $z$ direction). The matrix form for pure states is obtained by taking the outerproduct of the pure state vector: $\op{\psi}{\psi}$.} 

Consider a quantum bit, subjected to continuous weak measurement of the Hermitian operator $\sigma_{z}$ ($z$ component of angular momentum)  and feedback. The goal of the feedback is to take the initial state $\ket{0}$  and control it to the orthogonal state $\ket{1}$ in a time optimal manner. Intuitively, this requires a control rotation about the $Y$ axis. See \frf{Fig1} for a representation of this control problem on the Bloch sphere \cite{nielsen2000qca}.

\begin{figure}
\begin{center}
\includegraphics[width=0.4\hsize]{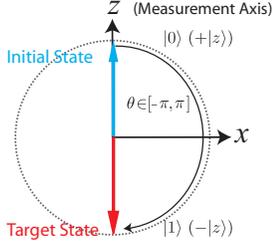}
\caption{The Bloch sphere with a graphical depiction of our control problem. We start in the plus eigenstate of the observable $\sigma_{z}$ and rotate to the orthogonal state $-\ket{z}$. {The controlled rotation axis is out of the page.} 
\label{Fig1} 
}
\end{center}
\end{figure}

{The model for such a system is given by the Stochastic Master Equation (SME) \cite{brun,steck,WisMil10}:
\begin{align}\label{SMEprob}
d\rho =& -i dt\, \smallfrac{1}{2}\alpha(t) [\sigma_{y},\rho] +2\gamma dt\, \Dc{\sigma_{z}}\rho\nonumber \\
&\,\,+\sqrt{2\gamma}dW\,\Hc{\sigma_{z}}\rho
\end{align}
The {\em measurement strength} $\gamma$ determines the rate at which measurement extracts information about the observable $\sigma_z$. In the equation above: 
$\alpha(\cdot)$ denotes the  control signal\footnote{Physically this could arise from applying classical time dependent fields to a quantum system, which is standard in open loop quantum control.}; \lq$\dg$\rq\, denotes the adjoint of an operator; $dW$ is the innovation process~\cite{wong1985stochastic}; $[A,B]$ is the commutator; and   $\Dc{A} \rho \equiv A\rho A\dg -\half (A\dg A \rho + \rho A\dg A),$ $\Hc{A} \rho \equiv  A\rho +\rho A\dg - \tr{(A\dg+ A )\rho}\rho$ are superoperators \cite{WisMil10}.

From \erf{SMEprob} we may calculate the stochastic differential equations (SDE) for the Bloch components using the relation $d k = \tr{d\rho\,\sigma_{k}}$ (where $k\in {x,y,z}$). Next we assume that the available control is equal in strength (isotropic) about all axes $(x, z, y)$. This drastically simplifies the problem, as we can now exploit this symmetry to consider control about just one axis (the $y$ axis), so the dynamics are restricted to a single plane (the $x$-$z$ plane).  We further simplify the problem by transforming to polar coordinates 
$z = R \cos \theta,\,\, x = R \sin\theta$, where $\theta = \tan{(x/z)}^{-1}$, by applying the  It\={o} rules on the equations for the  Bloch components. For initially pure states (i.e., $\tr{\rho^{2}}=1$) the Bloch vector is confined to the surface of the sphere, i.e. $R=\sqrt{x^{2}+z^{2}}=1$ and thus our control problem reduces to the stochastic differential equation (SDE) for an angle:}
\begin{align}
d\theta \aeq \alpha(t)\, dt - 2 \gamma \sin(2 \theta) dt - 2 \sqrt{2 \gamma} \sin(\theta) dW. \label{\eqnname syseqn}
\end{align}
where $\theta \in[-\pi,  \pi]$.

%
%
%
%
The first term in \erf{\eqnname syseqn} is the control signal applied. To ensure that the control problem is well posed we apply a bounded strength control, i.e. the controls are constrained to a   compact set $\controlValueSet:=[-\controlBound, \controlBound]$. In addition we require that 
$\Omega > 2 \gamma$ (for reasons required in Thm.~\ref{\thmname relatecostfns});  Note that this is a sufficient condition but may not be necessary. 
The class of piecewise continuous control signals that take up their values from $\controlValueSet$ is denoted by  $\controlSet$. This set $\controlSet$ is the set of signals which are progressively measurable with respect to the filtration of the random process.
The second term in \erf{\eqnname syseqn} represents the measurement back action.
  The final term is the innovation term arising from measurements.

The solution to the SDE Eq.~\eqref{\eqnname syseqn} at any time $t\in [t_{0},\infty)$ starting from a point $\thn{0}$ (at a time $t_{0}$) and using a control {strategy} $\alpha \in \controlSet$, 
is denoted by $\solnot(t;\alpha,t_{0},\thn{0})$. Note that this is a random variable whose value depends on an underlying sample space.  If the arguments used in this expression are clear from context, we represent   the solution at time $t$ using the  simplified notation $\thn{t}$.

\section{Two optimal control problems}
We consider two ways to formulate the cost function for time optimal rotation.  The first possible formulation
is the expected hitting time i.e. the expectation of the times at which a trajectory hits  the target state $\targetSet$ (say $\pm \pi$).  The second formulation, is the shortest time at which average or expected trajectory first reaches $\targetSet$ - termed the {\em expected trajectory reaching time}.  We obtain the Hamilton-Jacobi-Bellman equations to be solved for these problems. Numerical solutions to this HJB equation can be obtained via standard techniques such as the value iteration methods.

\subsection{{Mean (discounted) hitting time}}\label{\secname dismeantimetotheta}
The {original} problem of interest  is to determine the average hitting time  to the target angle of $\pi$.  However it turns out that the Hamilton-Jacobi-Bellman equation that arises from this optimal control  problem is a degenerate elliptic PDE. As the existence of classical solutions to this equation are not guaranteed, we formulate an alternate cost function -- one for which  the existence and uniqueness of generalized (weak) solutions can be rigorously shown. This modified cost is the expected discounted hitting time to the target set  $\targetSet_{e}:=\pm \pi$. Thus the objective of the control problem is to control the system in order to minimize the expected discounted mean time to hit the target. 
We note that the solution to this problem is potentially different from the solution to the original un-discounted problem.  The proofs of the uniqueness and existence of the viscosity solutions  to the undiscounted case would require results applicable to degenerate HJB PDEs specific to the  undiscounted cost function - a result that we are currently unaware of.
 
 Consider the optimal cost function defined over the set $\thetaset:= (-\pi, \pi)$, which has the form:
\begin{align}
\ocostfnone (\thn{0}) \aeq \inf_{\notncsig \in \controlSet}  \costfnone (\thn{0},\notncsig), \label{\eqnname meanhittingtimevalfn}
\end{align}
where
\begin{align}
\costfnone (\thn{0},\notncsig) \aeq  \mathrm{E} \Big[ \int_{0}^{\tau ^{\notncsig}_{\targetSet_{e}}{(\thn{0})}}{\exp{\{-\lambda s\}}\,ds}  \Big], \nonumber\\
\tau ^{\notncsig}_{\targetSet}(\thn{0}) &:= \inf \left\{t | \solnot(t; \notncsig,0, \thn{0}) \in \targetSet \right\}.\label{\eqnname discount}  
\end{align}
 The parameter $\lambda$ ($>0$) in \erf{\eqnname discount}   is called a discount factor.
In order to obtain the optimal control strategy to this problem, we apply the dynamic programming method \cite{bellman2003dp,bertsekas1995dpa} from optimal control theory. This yields the associated  Hamilton-Jacobi-Bellman equation, for the cost function $\ocostfnone$, of the form
\begin{align}
\sup_{\notncsigval \in \controlValueSet}  \left \{-1+ \lambda \phi - L^{\notncsigval }[\phi ] (y) \right\} =0, \quad \forall y \in \thetaset \label{\eqnname hjbdef1}
\end{align}
with boundary conditions $\phi(\targetSet_{e}) = 0$. The differential operator $L^{v}[\phi](y)$ in \erf{\eqnname hjbdef1} is defined as
\begin{align}
 L^{\notncsigval }[\phi ] (y) \adefeq  \mathbf{b}(y,v) \spd{\phi}{\theta} \Bigg |_{\theta = y}+ \frac{1}{2} {\boldsymbol\sigma}^{2}(y) \dpd{\phi}{\theta}\Bigg |_{\theta = y} \label{\eqnname loperatordefn},
\end{align}
and is understood as the generator of the It\^o diffusion process Eq.\eqref{\eqnname syseqn}. The coefficients ${\bf b}, {\boldsymbol \sigma}$ may be obtained  from the relevant SDE \erf{\eqnname syseqn} which has the form  $dx = {\bf b}(x,v)dt + {\boldsymbol \sigma}(x) dW$. Hence, we find that ${\bf b}(x,v) :=  v - 2 \gamma \sin(2 x)$ and  ${\boldsymbol\sigma}(\theta) := 2 \sqrt{2 \gamma} \sin(\theta)$.
Note that  Equation (\ref{\eqnname hjbdef1}) (with $\lambda = 0$) takes the form of a degenerate elliptic PDE \textit{irrespective of whether $\lambda$ is $> 0$ or $= 0$}; this is because the second order partial derivative term $\sigma(\cdot)$ is zero at some points in the domain. \blk Hence the positivity condition on the coefficient of the second order derivative, which is  sufficient  for classical solutions to this equation to exist, does not hold \cite{wong1985stochastic,fleming2006cmp}  \footnote{Unsurprisingly,  the {sufficiency} condition for the Fokker-Planck equation \cite{gardiner1985handbook} to have a smooth solution also does not hold.}.  Physically this degeneracy corresponds to the presence of a symmetry of revolution around that point. In order to analyze this situation 
we  study the solution to this problem via the notion of generalized (viscosity) solutions to this PDE.  The desired uniqueness and existence result for the hitting time function proceed as follows.

\subsubsection{Viscosity solution for a discounted hitting time problem} 
%
It turns out that the HJB equation associated with this optimal control problem has the form: $F(x,u,Du,D^{2} u) = 0$, with boundary conditions $\phi(\targetSet) = 0$. For the system and cost function under consideration, $F(\cdot)$ is defined as
\begin{align}
F(x,u,p,M)\!\adefeq - 4 \gamma [\sin(x)] ^{2} M +  \nonumber\\
&\,\sup_{\notncsigvalalp \in \controlValueSet} \Big\{ \!-\!\Big[\notncsigvalalp - 2 \gamma \sin(2 x)\Big] p + \lambda u -1 \Big\}\label{\eqnname hjbformeanpathwithomegaB}.
\end{align}
The second term on the right hand side of the expression above is termed the Hamiltonian and can be represented by a function of the form $H(x,u,p)$.
Note that in this case, the HJB equation is a degenerate  non-linear elliptic PDE hence classical solutions may not exist. Therefore it is necessary to understand the solution to this equation in a weak / viscosity sense.
\subsubsection{Existence and uniqueness of the viscosity solution} \label{\appendname viscsolnfordiscsoln}
The following result yields the uniqueness of the hitting time function.
\begin{theorem} \label{\thmname uniquofdiscountedvs}
The value function is
 the unique continuous viscosity solution of the HJB equation: $-\half \sigma^{2}(x) D^{2}u + H(x,u,Du) = 0, \,\,  x \in G$, in $G$  with boundary condition $u(\partial G) = 0$.
\end{theorem}
\begin{proof}
This result follows from  Theorem 4.1 in \cite{barles1995dirichlet}. 
\end{proof}
%


%
\subsection{Expected trajectory {reaching} time}\label{\secname averagetheta} 
The cost function to be minimized for this problem takes the form
\begin{align}
\costfntwo (\thn{0},\notncsig) \adefeq \inf \left\{ t | \mathrm{E}\big[ \solnot(t; \notncsig, 0, \thn{0}) \big] \in \targetSet \right\} \nonumber \\
&=  \inf \left\{ t | \big{|} \mathrm{E}\big[ \solnot(t; \notncsig, 0, \thn{0}) \big]\big{|} = \pi \right\}.\label{\eqnname costfntwodef}
\end{align}
The optimal cost function in this case is:
\begin{align}
\ocostfntwo(\thn{0}) \adefeq \inf_{\notncsig \in \controlSet} {\costfntwo(\thn{0}, \notncsig)} \label{\eqnname ocostfntwodef}.
\end{align}
Now, the form of the  HJB equations for the cost function above with terminal constraints on the expectation $\Eavg{\cdot}$ has a form that does not permit {analytical}  analysis. This is due to the fact that as the underlying SDE, \erf{\eqnname syseqn}, is nonlinear. Consequently the fundamental quantity of interest $\Eavg{\solnot(\tau;\alpha,t_{0},\thn{0})}$ does not appear to have a closed form solution. 
Nevertheless it is possible to solve the control problem for the optimal strategy by formulating an associated problem, whose solution leads to the solution of \erf{\eqnname ocostfntwodef}. 
This is achieved as follows. 
Consider the problem of determining  the control policy that maximizes $|\Eavg{\solnot(T; \alpha,t_{0},\thn{0})}|$  at a fixed time $T$. This can be easily seen to be equivalent to choosing the maximum of the terms ${M_{1}}(\thn{0})$ and ${M_{2}}(\thn{0})$ where these  functions arise from the  optimal control problems: 
\begin{subequations}\label{\eqnname m1eqn}
\begin{align}
 M_{1}(\thn{0})\adefeq \sup_{\alpha \in \controlSet} \Eavg{\solnot(T; \alpha,t_{0},\thn{0})},\\
M_{2}(\thn{0}) &:= -\inf_{\alpha \in \controlSet}\Eavg{\solnot(T; \alpha,t_{0},\thn{0})}.
\end{align}
\end{subequations}
Formally both of these are Mayer type optimal control problems (with zero running cost, fixed terminal penalty and a fixed time horizon $T>0$). Specifically, the equations above have the form:
\begin{align}
 M_{1}(\thn{0})= \oaltcostfn^{a}_{t,T} (\thn{0})  &= \sup_{\notncsig \in \controlSet} 
  \mathrm{E} \big[\solnot(T; \notncsig,t,\thn{0})\big] \nonumber \\
&=  \sup_{\notncsig \in \controlSet} \mathrm{E} \left [ \int_{t}^{T} 0 \,ds  + \thn{T} \right ]. \label{\eqnname oaltcostfn1def}
\\
 M_{2}(\thn{0})= \oaltcostfn^{b} _{t,T} (\thn{0})  &=  \sup_{\notncsig \in \controlSet} 
 \mathrm{E} \big[-\solnot(T; \notncsig,t,\thn{0})\big]  \nonumber\\
&=  \sup_{\notncsig \in \controlSet} \mathrm{E} \left [ \int_{t}^{T} 0 \,ds  - \thn{T} \right ].\label{\eqnname oaltcostfn2def}
\end{align}

We now describe the  result that links these cost functions $\oaltcostfn^{a}, \,\oaltcostfn^{b}$,  with the original cost function of interest $\ocostfntwo$ (Eq.~\eqref{\eqnname ocostfntwodef}).
\begin{theorem}\label{\thmname relatecostfns}
\begin{align}
\ocostfntwo(\thn{0}) = \inf \left \{ T \big| \max\{\oaltcostfn^{a} _{0,T}(\thn{0}), \oaltcostfn^{b} _{0,T}(\thn{0})\} > \pi \right\} \label{\eqnname diffTvaluesforRS}.
\end{align}
\end{theorem}
\begin{proof}
We begin by noting that the RHS of Eq.~\eqref{\eqnname diffTvaluesforRS} can be written as follows 
\begin{align}
&\max\{\sup_{\notncsig \in \controlSet}[\eop(\theta(T ; \notncsig, 0, \thn{0}))], \sup_{\notncsig \in \controlSet}[\eop(-\theta(T ; \notncsig, 0,\thn{0}))]\}\noncr
&= \sup_{\notncsig \in \controlSet}[{\max(\eop(\theta(T ; \notncsig, 0, \thn{0})), -\eop(\theta(T ; \notncsig, 0, \thn{0})))}] \noncr
&=  \sup_{\notncsig \in \controlSet}[|\eop[\theta(T ; \notncsig, 0, \thn{0})]|]
\end{align}
Hence the statement of this theorem is equivalent to demonstrating that 
\begin{align}
\{T| \inf_{\notncsig \in \controlSet}\{ t_{\notncsig}\big{|} |\eop [\theta(t_\notncsig;\notncsig,0,\thn{0})]| \geq \pi \}<T \}\nonumber \\
= \{ T|  \sup_{\notncsig \in \controlSet} [|\eop [\theta(T;\notncsig,0,\thn{0})]|] > \pi \},\label{\eqnname L1}
\end{align}
as the desired result follows immediately from this (by taking an infimum over both sets). 
Thus the proof proceeds via two steps.
\\\textbf{Stage 1}:  We demonstrate that the LHS $\supseteq$ RHS i.e., any element $T$ of the set on the right also belongs to the set on the left.\\
Given $T_1 \in \{ T|  \sup_{\notncsig \in \controlSet} [|\eop [\theta(T;\notncsig,\thn{0})]|] > \pi \}$
it follows that  
\begin{align}
\exists \notncsig_1 \in \controlSet  \,\,\text{s.t } \quad
|\eop [\theta(T_1;\notncsig_1,0,\thn{0})]| > \pi
\end{align}
Therefore $\inf_{\notncsig \in \controlSet}\{ t_{\notncsig}\big{|} |\eop [\theta(t_\notncsig;\notncsig,0,\thn{0})]| \geq \pi \}<T_1 \} $. This implies that 
\begin{align}
T_1 \in \{T| \inf_{\notncsig \in \controlSet}\{ t_{\notncsig}\big{|} |\eop [\theta(t_\notncsig;\notncsig,0,\thn{0})]| \geq \pi \}<T \},
\end{align}
thereby proving the first part.
%
\\
\textbf{Stage 2:} We demonstrate that the RHS $\supseteq$ LHS.\\
Let $T_1 \in \{T| \inf_{\notncsig \in \controlSet}\{ t_{\notncsig}\big{|} |\eop [\theta(t_\notncsig;\notncsig,0,\thn{0})]| \geq \pi \}<T \} $.
 Hence
\begin{align}
\exists t_1 < T_1,\,\notncsig_1 \in \controlSet  \,\,\text{s.t}\quad |\eop [\theta(t_1;\notncsig_1,0,\thn{0})]| = \pi.
\end{align}
Without loss of generality we consider the case $\eop [\theta(t_1;\notncsig_1,0,\thn{0})]=\pi$. The alternative case follows almost directly. 
Now, by applying a control $+\Omega$  for all $t > t_1$ we have from the solution of the SDE that:
\begin{align}
\eop[\theta(T,\notncsig,0,\thn{0})] - \eop[\theta(t_1,\notncsig_1,0,\thn{0})]\nonumber\\
 = \eop\left \{\int_{t_1}^{T}{[\notncsig(t) - 2 \gamma \sin(2 \theta(t))] dt} \right\}.
\end{align}
Note that due to the assumption on the control signal bound ($\Omega > 2 \gamma$), it follows that
\begin{align}
\eop[\theta(T,\notncsig,0,\thn{0})] - \eop[\theta(t_1,\notncsig_1,0,\thn{0})]  > 0,
\end{align}
Hence $T>t_1$ and therefore $\sup_{\notncsig \in \controlSet} [\eop [\theta(T;\notncsig,0,\thn{0})]] > \pi  $, leading to the fact that $T \in RHS$. This proves the second part.
%
%
From the two stages above, the result follows. 
%
\end{proof}

This result indicates an approach to obtain the desired cost function $\ocostfntwo$; 
We first determine the solution to the control problems for $\oaltcostfn^{a}_{0,T}$, $\oaltcostfn^{b}_{0,T}$ given by Eqns.~\eqref{\eqnname oaltcostfn1def}, \eqref{\eqnname oaltcostfn2def} for a particular value of $T$. Then we can use  progressively smaller values of this terminal time $T$, in order to obtain the desired value of $\ocostfntwo$ via  Eq.~\eqref{\eqnname diffTvaluesforRS}.
Now, 
to obtain the solutions to the optimal control problems $\oaltcostfn^{a}_{(\cdot)}$, $\oaltcostfn^{b}_{(\cdot)}$   we apply the dynamic programming method  to yield the Hamilton-Jacobi-Bellman equation (HJB). For the sake of brevity, we describe the procedure for $\oaltcostfn^{a}_{(\cdot)}$, since the approach for $\oaltcostfn^{b}_{(\cdot)}$ then  follows immediately. 
Given any function\footnote{the notation denotes $C^{1}$ in time and $C^2$ in space.} $\phi  \in C^{1,2}(\Largethetaset)$ where $\Largethetaset := [0,T] \times \mathbb{R} $ we formulate the associated HJB equation  for the optimal control problem \eqref{\eqnname oaltcostfn1def} as
\begin{align}
\sup_{\notncsigval \in \controlValueSet}  \left \{\spd{\phi}{t} + L^{\notncsigval }[\phi] (y) \right\} =0, \quad \forall y \in \mathbb{R}, \quad t \in [0, T] \label{\eqnname hjbaltcostfn}.
\end{align}
The differential operator $L^{v}[\phi](y)$ in the equation above is defined as before. This HJB equation is 
 solved over the domain $\mathbb{R}$. Note that this domain is different from $[-\pi, \pi]$ since the objective in the control problem is to maximize the final angle. Furthermore, we observe  that at the point $\theta = 0$ the  control function is non-unique due to the symmetric nature of the problem. Hence  there would exist two alternatives for the optimal control at $\theta = 0$. We omit this analysis  for the sake of brevity.
The boundary condition for the PDE in  Eq.~\eqref{\eqnname hjbaltcostfn} is: 
\begin{align}
\phi(T, y) \aeq y, \quad \forall y \in \mathbb{R} \label{\eqnname altcontrolprobBC}.
\end{align}
As this HJB equation is degenerate parabolic we indicate the existence and uniqueness of the corresponding viscosity solution  as follows.

\subsubsection{Viscosity solution for a finite time horizon problem}\label{\appendname viscsolnfinitehorizonprob}
Consider a  system evolving according to the dynamics $
dx(s) = f(s,x(s),v(s)) ds + \sigma(x(s)) dw(s), $ for $s \in [t_{0}, T]$. 
The expected trajectory hitting time problem for this system can be recast as a optimal control problem over a finite time horizon, with a cost function of the form
\begin{align}
V(t, \thn{0})  \adefeq 
\sup_{\notncsig \in \controlSet} {\mathrm{E} \left [ \int_{t}^{T} L(s,x,v) \,ds  + \Psi({\thn{}(T)})\right ]}.\label{\eqnname valfnFinitetimeExpectedtrajectory}
\end{align}

The system dynamics and the value function are defined over the cylinder $Q_{0}~:= ~[t_{0}, T) \times \mathbb{R}^{n}$.
The HJB equation associated with this optimal control problem has the form 
\begin{align}
-\spd{V}{t} &+ H(t,x,D_{x}V, {D_{x}}^{2}V) = 0, \qquad  (t,x) \in Q_{0} \\
\text{where} &\quad H(t,x,p,A):=  \nonumber\\
&\sup_{\alpha \in \controlValueSet} \Big\{-f(t,x,\alpha) p -  \half \sigma^{2}(x) A - L(t,x,\alpha)\Big\} \label{\eqnname hjbforfixedtimehorizon}
\end{align}
with the boundary condition $V(T,x) = \Psi(x), \quad \forall x \in \mathbb{R}^{n}$.

The following result helps ensure the desired properties of the viscosity solution. 
\begin{theorem}\label{\thmname uniquenessofundiscfixedtimecostfn}
The value function Eq.~\eqref{\eqnname valfnFinitetimeExpectedtrajectory} is the unique, uniformly continuous viscosity solution to the HJB equation \eqref{\eqnname hjbforfixedtimehorizon}. 
\end{theorem}
\begin{proof}
The result follows from \cite[Ch.\ V \S9 and Theorem 9.1]{fleming2006cmp}. 
\end{proof}
\section{Simulations}\label{\secname Simulations}  
{In this section we describe  numerical solutions and simulation results for the problems described in Sections \ref{\secname dismeantimetotheta}, \ref{\secname averagetheta}.  We use a numerical approximation to  value iteration approach to obtain the solution to the PDEs in Eqns.~\eqref{\eqnname hjbdef1}, \eqref{\eqnname hjbaltcostfn}. In this method, the stochastic nature of the system dynamics is captured  as the transition probabilities of an approximating Markov chain model. The transition probabilities are then used form an iterative scheme that converges to yeild the value function i.e.,  the optimal cost function. For a detailed description of this approach and applications to quantum systems we refer readers to \cite{kushner1992nms, smjpra2008}.

{By this approach we obtain both the optimal cost function as well as the optimal control strategy. } These are depicted  in  \frf{dischittingtimevalfn} for the case of the mean hitting time; as the corresponding plots for the expected trajectory look very similar we do not plot them. The optimal control strategy in order to rotate from $\ket{0}$ to $\ket{1}$ is to set the Hamiltonian control $v$ to $+ \controlBound$ and this is consistent with the intuitively correct,  approach, viz.,  in order to reach an angle of $\pi$ as soon as possible we must apply the maximum control in that direction.

\begin{figure}[h!]
\begin{center}
\leavevmode \includegraphics[width=1\hsize]{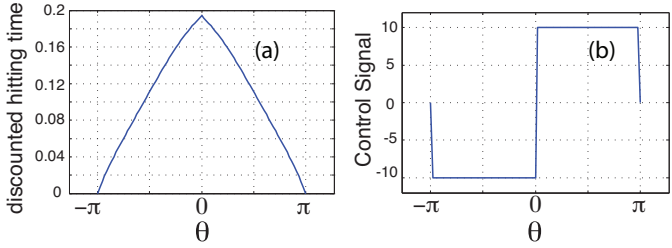}
\caption{(a) The mean hitting time to $\pm \pi$ starting from  $\theta(0) = 0$, with $\Omega=10,\gamma=1$ and $\lambda = 3$. (b) the optimal control function for the same parameter values. Note the point $\theta = 0$ at which the cost function is not differentiable in the spatial term. This reaffirms the fact that the optimal cost function  is not a classical solution of the HJB equation for our problems.
\label{dischittingtimevalfn} 
}
\end{center}
\end{figure}

 In \frf{Fig6} we depict the mean hitting time to $\pm \pi$, for the system described by \erf{\eqnname syseqn}, as a function of the control strength. Interestingly there is a {\em only} a small difference between behaviour of the mean hitting time and the expected trajectory hitting time which disappears as $\Omega \rightarrow \infty$. In this limit  the mean hitting time asymptotes to the Hamiltonian evolution hitting time as does the expected trajectory hitting time.  This is unlike the results obtained in the rapid purification literature where Wiseman and Ralph's \cite{wiseman2006reconsidering} protocol outperforms Jacobs \cite{jacobs2003project} protocol by a factor of two in the regime of strong control (the limiting case)\cite{wiseman2008optimality}.  
 The simulations in the \frf{Fig6} were obtained via standard Monte Carlo simulations.
\begin{figure}[h!]
\begin{center}
\leavevmode \includegraphics[width=0.9\hsize]{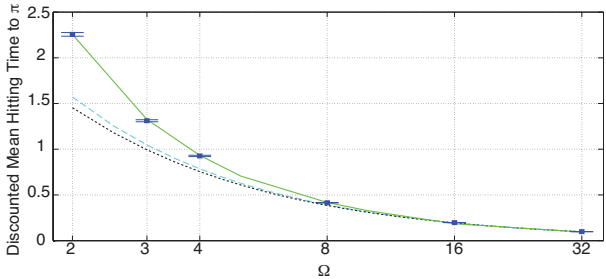}
\caption{The discounted mean hitting time to $\pi$ as a function of the control strength $\Omega$ and $\gamma=1$. In all simulations the discounting factor is $\lambda = 0.1$. The squares with error bars are the results of stochastic simulations (ensemble size is 5000). The solid line is a plot of the solution to \erf{\eqnname hjbdef1} i.e., the optimal average  hitting time. The dashed line is the time at which the discounted Hamiltonian evolution hits $\pi$. The dotted line is Hamiltonian evolution hitting time.}
\label{Fig6} 
\end{center}
\end{figure}

 We have also considered the case of optimal rotation control between different initial and final states. For example the discounted hitting time between $\pm\ket{x}$ (where $\pm\ket{x}$ are eigenstates of $\sigma_{x}$) for weak control is less than the time between $\pm\ket{z}$.  This is also true for the expected trajectory hitting time. The intuition for this effect is at $\theta=\{0,\pm \pi\}$ \erf{\eqnname syseqn} is `like' an attractor because the measurement projects the state into an eigenstate of the measured observable.
 } 

\section{Bounds on the optimal solutions for state preparation from a mixed state}\label{\secname bounds}
{In this section we obtain a bound on the minimum time taken to evolve from a maximally mixed state to  any desired pure state. We first bound the mean hitting time to prepare the target state specified by the pair ($\theta$, $P$), where $P=\tr{\rho^{2}}$ is its purity, starting from the maximally mixed state (i.e. $P=1/2$). By monitoring the system continuously a purity $P=1-\epsilon$ is achieved, on average, in time $\tau_{WR}(P)~= ~(\sqrt{2P-1}~\,\mathrm{tanh}^{-1} \sqrt{2P-1})/8\gamma,$ (using the Wiseman-Ralph protocol which is time optimal in the mean hitting time sense \cite{wiseman2008optimality,belavkin2009dynamical}). 
To obtain an upper bound on the time taken to rotate this eigenstate to any other state we consider the worst case: the target state is orthogonal to the initial state. 
For $|\Omega|\gg |\gamma|$ we bound the rotation time by $\tau_{r}(\theta)$, i.e., the solution of {\erf{\eqnname discount}}. For $\epsilon\ll 1$ we may write the upper bound on the optimal time for state preparation as $\tau_{UB}= \tau_{WR}(P) + \tau_{r}(\theta)$, where this bound is understood in the mean hitting time sense.

In the case of the minimum time for the expected trajectory, Jacobs' purification scheme \cite{jacobs2003project} is time optimal \cite{wiseman2008optimality}. However the control rotations were instantaneous impulse control which is different from our finite strength description. Nevertheless we calculate a lower bound on the minimum time using his result $t_{\rm J}(P)= -\ln{(2P-2)}/8\gamma$. Using the worst case scenario 
the bound on the rotation time is $t_{r}(\theta)$, i.e., the solution of  Eq.~\eqref{\eqnname costfntwodef}. Consequently the lower bound on the realistic time optimal protocol is $t_{LB}= t_{\rm J}(P) +t_{r}(\theta)$, which is understood in the expected trajectory hitting time sense. We obtain an upper bound by assuming that, in the worst case, the optimal bounded strength control performs better than the Wiseman-Ralph protocol i.e., $\tau_{WR}(P)$ provides an upper bound on the true optimal time $t_{*}(P)$. In this case the upper bound  is $t_{UB}= t_{\rm WR}(P) +t_{r}(\theta)$. Thus $t_{LB}\le t_{*}(\theta, P)\le t_{UB}$.
}

\section{Conclusion and Open Problems for Future Work}\label{\secname conc}
In continuing with our investigation  of weak solutions there  is a need  for deeper analysis to prove the existence and uniqueness in the limit of the discount factor tending to zero. Moreover our work here suggests an exciting possibility: of being able to improve the optimal control strategy by using stages of pure Hamiltonian evolution (where we turn the measurement off) and  other periods where we use both measurement and feedback.

{\em Acknowledgements} The authors would like to thank the referee of a previous version of the manuscript for insightful feedback and Matthew James for helpful comments. JC is supported by NSF Grants PHY-0903953, PHY-1212445, and
PHY-1005540 and by ONR Grant No. N00014-11-1-0082.

{

\appendices

\bibliographystyle{IEEEtran}

\end{document}